\documentclass[]{interact}

\usepackage{epstopdf}
\usepackage[caption=false]{subfig}
\usepackage{color}
\usepackage{enumerate}
\usepackage[numbers,sort&compress]{natbib}
\bibpunct[, ]{[}{]}{,}{n}{,}{,}
\makeatletter
\def\NAT@def@citea{\def\@citea{\NAT@separator}}
\makeatother

\theoremstyle{plain}
\newtheorem{theorem}{Theorem}[section]
\newtheorem{lemma}[theorem]{Lemma}
\newtheorem{corollary}[theorem]{Corollary}
\newtheorem{proposition}[theorem]{Proposition}

\theoremstyle{definition}

\newtheorem{example}[theorem]{Example}

\theoremstyle{remark}
\newtheorem{remark}{Remark}

\usepackage{mathrsfs}        
\usepackage{amsmath}         
\usepackage{bbding}          
\usepackage[misc]{ifsym}     

\begin{document}

\articletype{ARTICLE TEMPLATE}
\title{Discussion on the Leibniz rule and Laplace transform of fractional derivatives using series representation}
\author{
\name{Yiheng Wei\textsuperscript{a}, Da-Yan Liu\textsuperscript{b}, Peter W. Tse\textsuperscript{c} and Yong Wang\textsuperscript{a} \thanks{CONTACT Y. Wang. Email: yongwang@ustc.edu.cn}}
\affil{\textsuperscript{a}Department of Automation, University of Science and Technology of China, Hefei, 230026, China; \textsuperscript{b}INSA Centre Val de Loire, Universit\'{e} d'Orl\'{e}ans, PRISME EA 4229, Bourges Cedex, 18022, France  ; \textsuperscript{c}Department of Systems Engineering and Engineering Management, City University of Hong Kong, Hong Kong, 999077, China.
}}
\maketitle

\begin{abstract}
Taylor series is a useful mathematical tool when describing and constructing a function. With the series representation, some properties of fractional calculus can be revealed clearly. On this basis, the Lebiniz rule and Laplace transform of fractional calculus is investigated. It is analytically shown that the commonly used Leibniz rule cannot be applied for Caputo derivative. Similarly, the well-known Laplace transform of Riemann--Liouville derivative is doubtful for $n$-th continuously differentiable function. After pointing out such problems, the exact formula of Caputo Leibniz rule and the explanation of Riemann--Liouville Laplace transform are presented. Finally, three illustrative examples are revisited to confirm the obtained results.
\end{abstract}


\begin{amscode}
Primary: 26A33, Secondary: 30K05, 44A10, 65L05
\end{amscode}

\section{Introduction}\label{Section 1}
\setcounter{section}{1} \setcounter{equation}{0}
The subject of fractional calculus (that is, calculus of integrals and derivatives of any arbitrary real or complex order) has gained considerable popularity and importance during the past three decades or so, due mainly to its demonstrated applications in numerous seemingly diverse and widespread fields of science and engineering. It does indeed provide several potentially powerful tools for modelling and controlling those practical plants with history dependent and global correlative properties. As for the recent relevant works, the readers can refer to the excellent papers \cite{West:2016Book,Boyadzhiev:2018ITSF,Sun:2018CNSNS,Wei:2019NODY} and the references therein.

Fractional derivatives lead to a lot of unusual properties. For example, all the well-known fractional derivatives involving Riemann--Liouville, Caputo and Gr\"{u}mwald--Letnikov definitions violate the usual form of the Leibniz rule ${}_a^{}D_t^\alpha \{f(t)g(t)\}={}_a^{}D_t^\alpha f(t)g(t)+f(t){}_a^{}D_t^\alpha g(t)$ \cite{Liu:2015SP,Liu:2017Auto}. Though many references \cite{Srivastava:1993ITSF,Tarasov:2013CNSNS,Tarasova:2016FDC,Tarasov:2016JCND} confirm this unusual property and present the corresponding Leibniz rule, the questionable rule was still used widely \cite{Chen:2014CNSNS,Mujumdar:2015TM}. Particularly, the original study of Leibniz rule for fractional derivatives can date back to 1832 when Liouville focused on the issue and developed a correct Leibniz formula in the form of infinite series summation \cite{Oldham:1974Book,Sayevand:2018CNSNS}. Generalizations of the Leibniz rule for fractional derivatives are also derived successively by Osler in \cite{Osler:1970JMA,Osler:1971AMM,Osler:1972MC}. On this basis, Tremblay improved the fractional Leibniz rule and its integral analogue further \cite{Tremblay:2013ITSF}. These Leibniz rules are especially useful for the evaluation of fractional derivative of a product function, such as, the Lyapunov function. Nonetheless, on the issue of Leibniz rule, only Riemann-Liouville definition was considered.

Under Caputo definition, to the best of our current knowledge, the related Leibniz rule in infinite series form has been adopted in many papers \cite{Aguila:2014CNSNS,Aguila:2015CNSNS,Hua:2018NEU}. 
In fact, the adopted Leibniz rule in \cite{Aguila:2014CNSNS} holds for Riemann-Liouville derivative or even Riemann-Liouville integral and it has been highly accepted and widely applied. However, no literature could support the derivation process of such a rule and explain the peculiar items ${}_a^{\rm C}{\mathscr D}_t^{\alpha  - k}g\left( t \right)$ with $k >\alpha$. Fortunately, there are several attempts to define a new type of Leibniz rule for Caputo fractional derivative. For instance, Ishteva originated the related discussion and revealed the difference between Caputo and Riemann-Liouville derivative in Lemma 3 of a pioneering article \cite{Ishteva:2005MSRJ}. Afterwards, an equivalent expression of the Leibniz rule for Caputo derivative with the order belonging to $(0,1)$ was presented (see Remark 1 of \cite{Freed:2007FCAA} and Theorem 3.17 of \cite{Diethelm:2010Book}). From then, references \cite{Ortigueira:2015JCP,Teodoro:2019JCP} revisited the issue in but gave a questionable conclusion. Note that study on such Leibniz rule is still at its early stage. Additionally, as another typical application of fractional calculus, the Laplace transform of Riemann-Liouville derivative contains fractional derivatives as initial value which is not practical enough \cite{Zhou:2017ISA,Wu:2017AMC}. When the considered function is $n$-th continuously differentiable, the results will lead to singularity. When the initial instant is nonzero, the result becomes more complex \cite{Lundberg:2007ICSM,Belkhatir:2018SCL}. With the adoption of series representation \cite{Williams:2012FCAA,Wei:2017FCAAb,Wei:2018arXiv,Wei:2019CNSNS}, the Caputo Leibniz rule and Riemann--Liouville Laplace transform might be established anew. 

Bearing the above discussion in mind, this paper aims at addressing the applications of fractional series representation on Leibniz rule and Laplace transform. To this end, Section \ref{Section 2} presents basic definitions and the so-called series representation of fractional calculus. Section \ref{Section 3} shows the violation of fractional derivative on well-adopted Leibniz rule and Laplace transform and derives the correct counterparts. Section \ref{Section 4} validates the validity of the indicated fact. Section \ref{Section 5} draws the conclusions.

\section{Preliminaries}\label{Section 2}
\setcounter{section}{2} \setcounter{equation}{0}
The concept of fractional calculus has been known because of the development of the regular calculus, with the origin probably being associated with the thought of Leibniz about 300 years ago. Today, there are numerous different definitions related to fractional calculus, among which Riemann--Liouville and Caputo definitions are two of the most popular ones which have indeed played a striking role in science and engineering \cite{Podlubny:1999Book}.

In 1847, Riemann derived a definition for fractional integral as
\begin{equation}\label{Eq1}
{\textstyle {{}_a^{\rm R}}{{\mathscr I}_{t}^\alpha}{f\left( t \right)}  \triangleq \frac{1}{{\Gamma \left( \alpha  \right)}}\int_{{a}}^t {{{\left( {t - \tau } \right)}^{\alpha -1}}{ f\left( \tau  \right)}{\rm{d}}\tau },}
\end{equation}
which is commonly called Riemann--Liouville fractional integral, where $\alpha>0$ is the integral order, $a$ is the constant lower terminal and $\Gamma \left( \cdot \right)$ is the Euler gamma function.

In the light of such a fractional integral in (\ref{Eq1}), two fractional derivatives were established successively, i.e., Riemann--Liouville case (in 1872)
\begin{equation}\label{Eq2}
{\textstyle {}^{\rm R}_a{\mathscr D}_{t}^\alpha f\left( t \right) \triangleq \frac{{\rm{d}}^n}{{{\rm{d}}t^n}} {_a^{\rm R}{\mathscr I}_{t}^{n-\alpha}} f\left( t \right),}
\end{equation}
and Caputo case (in 1967)
\begin{equation}\label{Eq3}
{\textstyle {}^{\rm C}_a{\mathscr D}_{t}^\alpha f\left( t \right) \triangleq  {_a^{\rm R}}{\mathscr I}_{t}^{n-\alpha}\frac{{\rm{d}}^n}{{{\rm{d}}t^n}} f\left( t \right),}
\end{equation}
with $n\in\mathbb{N}_+$ and $\alpha\in(n-1,n)$.

Reference \cite{Podlubny:1999Book} shows that
\begin{equation}\label{Eq4}
{\textstyle
{}^{\rm R}_a{\mathscr I}_{t}^\alpha f\left( t \right)={}^{\rm R}_a{\mathscr D}_{t}^{-\alpha} f\left( t \right).}
\end{equation}
for any $\alpha>0$ and if the mentioned function $f(t)$ is $(n-1)$-times continuously differentiable and $f^ {(n)}(t)$ is integrable for $t\ge a$, the relationship between the Riemann--Liouville derivative and the Caputo one can be expressed as
\begin{equation}\label{Eq5}
{\textstyle
{}_a^{\rm{R}}{\mathscr D}_t^\alpha f\left( t \right) = {}_a^{\rm{C}}{\mathscr D}_t^\alpha f\left( t \right) + \sum\nolimits_{k = 0}^{n - 1} {\frac{{{f^{\left( k \right)}}\left( a \right)}}{{\Gamma \left( {k + 1-\alpha} \right)}}{{\left( {t - a} \right)}^{k - \alpha }}} .}
\end{equation}
for $\alpha\in(n-1,n)$ and $n\in\mathbb{N}_+$.

Based on the aforementioned definitions with $\alpha\in(n-1,n),~\forall n\in\mathbb{N}_+$, the following relations can be found smoothly \cite{Wei:2018arXiv}
\begin{equation}\label{Eq6}
{\textstyle\begin{array}{l}
 {\frac{{{{\rm{d}}^n}}}{{{\rm{d}}{t^n}}}f}\left( t \right)=\mathop {\lim }\limits_{\alpha  \to n} {}_a^{\rm{R}}{\mathscr D}_t^\alpha f\left( t \right)  = \mathop {\lim }\limits_{\alpha  \to n - } {}_a^{\rm{C}}{\mathscr D}_t^\alpha f\left( t \right),
  \end{array} }
\end{equation}
\begin{equation}\label{Eq7}
{\textstyle\begin{array}{l}
 {}_a^{}{\mathscr I}_t^{n}f\left( t \right)=\mathop {\lim }\limits_{\alpha  \to n} {}_a^{\rm{R}}{\mathscr I}_t^\alpha f\left( t \right) ,
\end{array}}
\end{equation}
which distinctly demonstrate that the fractional calculus could be regarded as the natural generalization of regular calculus. Nevertheless, such fractional derivatives with non-integer order are special integrals in relation to all historical data, which just forms the long memory characteristic of fractional derivatives. Occasionally, this property is also called as nonlocal characteristic or global correlation or history dependence. It is this characteristic that fashions the essential differences between the fractional calculus and the traditional one.

Before moving on, a key lemma will be given first. In light of this lemma, one can use the power functions $(t-a)^k$ and the integer derivatives values ${f^{\left( k \right)}}\left( t \right)$ or ${f^{\left( k \right)}}\left( a \right)$ to represent the fractional calculus, which brings great convenience in the subsequent theoretical analysis.

\begin{lemma}\label{Lemma 1}
(see \cite{Oldham:1974Book,Samko:1993Book,Podlubny:1999Book,Wei:2018arXiv}) 
If $f (t)$ can be expressed a Taylor series expanded at the initial instant $t = a$ and the series $f\left( t \right) = \sum\nolimits_{k = 0}^{ + \infty } {\frac{{{f^{\left( k \right)}}\left( a \right)}}{{k!}}} {\left( {t - a} \right)^k}$ is uniformly convergent on $[a, b]$, then for any $t\in[a,b]$, one has
\begin{equation}\label{Eq8}
{\textstyle{}^{\rm R}_a{\mathscr D}_{t}^\alpha f\left( t \right) = \sum\nolimits_{k = 0}^{+\infty}  {\frac{{{f^{\left( k \right)}}\left( a \right)}}{{\Gamma \left( { k + 1- \alpha } \right)}}} {\left( {t - a} \right)^{k - \alpha }},\alpha\in(-\infty,+\infty),}
\end{equation}
\begin{equation}\label{Eq9}
{\textstyle{}^{\rm C}_a{\mathscr D}_{t}^\alpha f\left( t \right) = \sum\nolimits_{k = n}^{+\infty}  {\frac{{{f^{\left( k \right)}}\left( a \right)}}{{\Gamma \left( { k + 1- \alpha } \right)}}} {\left( {t - a} \right)^{k - \alpha }},\alpha\in(n-1,n),n\in\mathbb{N}_+.}
\end{equation}
If $f (\cdot)$ is an analytic function in an interval $(a,b)$, which means that it can be represented as a convergent series $f\left( \tau  \right) = \sum\nolimits_{k = 0}^{ + \infty } {\frac{{{f^{\left( k \right)}}\left( t \right)}}{{k!}}} {\left( {\tau  - t} \right)^k}$, $\tau\in(a,t)$, then for any $t\in[a,b]$, one has
\begin{equation}\label{Eq10}
\begin{array}{l}
{}_a^{\rm{R}}{\mathscr D}_t^\alpha f\left( t \right)
= \sum\nolimits_{k = 0}^{ + \infty } {\left( {\begin{smallmatrix}
{\alpha }\\
{k }
\end{smallmatrix}} \right)\frac{{{f^{\left( k \right)}}\left( t \right)}}{{\Gamma \left( {k - \alpha  + 1} \right)}}{{\left( {t - a} \right)}^{k - \alpha }}},\alpha\in(-\infty,+\infty),
\end{array}
\end{equation}
\begin{equation}\label{Eq11}
\begin{array}{l}
{}_a^{\rm{C}}{\mathscr D}_t^\alpha f\left( t \right)
= \sum\nolimits_{k = n}^{ + \infty } {\left( {\begin{smallmatrix}
{\alpha  - n}\\
{k - n}
\end{smallmatrix}} \right)\frac{{{f^{\left( k \right)}}\left( t \right)}}{{\Gamma \left( {k - \alpha  + 1} \right)}}{{\left( {t - a} \right)}^{k - \alpha }}},\alpha\in(n-1,n),n\in\mathbb{N}_+.
\end{array}
\end{equation}
\end{lemma}

\section{Main Results}\label{Section 3}
\setcounter{section}{3} \setcounter{equation}{0}
This section focuses on the Leibniz rule and the Laplace transform of fractional calculus and develops some interesting observations.

\subsection{Leibniz Rule}
First, Leibniz rule for Riemann--Liouville derivative and integral will be introduced.

\begin{lemma}\label{Lemma 2}
(see \cite{Samko:1993Book,Podlubny:1999Book,Luque:1999ITSF,Babakhani:2017JPOA})
For any constant $\alpha\in(-\infty,+\infty)$, 
let $f(\tau)$ be a function whose Taylor series expansion exists at $t$, then
\begin{equation}\label{Eq12}
\begin{array}{l}
{}_a^{\rm{R}}{\mathscr D}_t^\alpha \left\{ {f\left( t \right)g\left( t \right)} \right\}
= \sum\nolimits_{j = 0}^{ + \infty } {\big( {\begin{smallmatrix}
\alpha \\
j
\end{smallmatrix}} \big){f^{\left( j \right)}}\left( t \right){}_a^{\rm{R}}{\mathscr D}_t^{\alpha  - j}g\left( t \right)},
\end{array}
\end{equation}
and if $g(\tau)$ can be expressed as a Taylor series at $t$, one has
\begin{equation}\label{Eq13}
\begin{array}{l}
{}_a^{\rm{R}}{\mathscr D}_t^\alpha \left\{ {f\left( t \right)g\left( t \right)} \right\}
= \sum\nolimits_{j = 0}^{ + \infty } {\big( {\begin{smallmatrix}
\alpha \\
j
\end{smallmatrix}} \big){g^{\left( j \right)}}\left( t \right){}_a^{\rm{R}}{\mathscr D}_t^{\alpha  - j}f\left( t \right)}.
\end{array}
\end{equation}
\end{lemma}

This lemma can be checked by the mentioned series in Lemma \ref{Lemma 1} and the facts $\sum\nolimits_{k = 0}^{ +\infty } {\sum\nolimits_{j = 0}^{ k } {} }=\sum\nolimits_{j = 0}^{ +\infty } {\sum\nolimits_{k = j}^{+\infty} }$, $\big( {\begin{smallmatrix}
\alpha \\
{i + j}
\end{smallmatrix}} \big)$ $\big( {\begin{smallmatrix}
{i + j}\\
j
\end{smallmatrix}} \big)= \big( {\begin{smallmatrix}
\alpha \\
j
\end{smallmatrix}} \big)\big( {\begin{smallmatrix}
{\alpha  - j}\\
i
\end{smallmatrix}} \big)$.
Setting $g(t)=t^m$, $m\in\mathbb{N}$, one can obtain the following corollary.
\begin{corollary} \label{Corollary 1}
For any suitable $f(t)$
\begin{equation}\label{Eq14}
\begin{array}{l}
{}_a^{\rm{R}}{\mathscr I}_t^{\alpha} \left\{ {f\left( t \right){t^m}} \right\} = \sum\nolimits_{k = 0}^m {{{\left( { - 1} \right)}^k}\left( {\begin{smallmatrix}
m\\
k
\end{smallmatrix}} \right)\frac{{\Gamma \left( \alpha + k  \right)}}{{\Gamma \left( {\alpha} \right)}}{t^{m - k}}{}_a^{\rm{R}}{\mathscr I}_t^{\alpha  + k}f\left( t \right)},
\end{array}
\end{equation}
\begin{equation}\label{Eq15}
\begin{array}{l}
{}_a^{\rm{R}}{\mathscr D}_t^{\alpha} \left\{ {f\left( t \right){t^m}} \right\} = \sum\nolimits_{k = 0}^m {{{\left( { - 1} \right)}^k}\left( {\begin{smallmatrix}
m\\
k
\end{smallmatrix}} \right)\frac{{\Gamma \left( -\alpha + k  \right)}}{{\Gamma \left( {-\alpha} \right)}}{t^{m - k}}{}_a^{\rm{R}}{\mathscr D}_t^{\alpha  - k}f\left( t \right)},
\end{array}
\end{equation}
hold for any $m\in\mathbb{N}$ and $\alpha>0$.
\end{corollary}
\begin{remark}\label{Remark 1}
Though the rule in (\ref{Eq12}) with $\alpha>0$ has been suggested by Liouville in 1832 as a generalization of the Leibniz rule for regular derivative, it is convenient to preserve Leibniz's name also in this case \cite{Luque:1999ITSF,Babakhani:2017JPOA}. With the help of (\ref{Eq6}) and $\Gamma(-k)=\pm\infty,~\forall k\in\mathbb{N}$, the $\alpha=n\in\mathbb{N}_+$ case follows from (\ref{Eq12}) and (\ref{Eq13}) immediately
\begin{equation}\label{Eq16}
{\textstyle
\frac{{{{\rm{d}}^n}}}{{{\rm{d}}{t^n}}}\left\{ {f\left( t \right)g\left( t \right)} \right\} = \sum\nolimits_{j = 0}^n {\left( {\begin{smallmatrix}
n\\
j
\end{smallmatrix}} \right){f^{\left( j \right)}}\left( t \right){g^{\left( {n - j} \right)}}\left( t \right)} ,}
\end{equation}
\begin{equation}\label{Eq17}
{\textstyle
\frac{{{{\rm{d}}^n}}}{{{\rm{d}}{t^n}}}\left\{ {f\left( t \right)g\left( t \right)} \right\} = \sum\nolimits_{j = 0}^n {\left( {\begin{smallmatrix}
n\\
j
\end{smallmatrix}} \right){g^{\left( j \right)}}\left( t \right){f^{\left( {n - j} \right)}}\left( t \right)} ,}
\end{equation}
which is just the traditional Leibniz rule.

Similarly, applying the formula (\ref{Eq4}), a special case, i.e., $\alpha=0$ can be derived from Lemma \ref{Lemma 2} as
\begin{equation}\label{Eq18}
{\textstyle
{}_a^{\rm{R}}{\mathscr D}_t^0\left\{ {f\left( t \right)g\left( t \right)} \right\} = {f\left( t \right)g\left( t \right)} .}
\end{equation}

When $\alpha<0$, the results in Lemma \ref{Lemma 2} can be simplified as
\begin{equation}\label{Eq19}
{\textstyle
{}_a^{\rm{R}}{\mathscr I}_t^{-\alpha} \left\{ {f\left( t \right)g\left( t \right)} \right\}= \sum\nolimits_{j = 0}^{ + \infty } {\big( {\begin{smallmatrix}
\alpha \\
j
\end{smallmatrix}} \big){f^{\left( j \right)}}\left( t \right){}_a^{\rm{R}}{\mathscr I}_t^{-\alpha  + j}g\left( t \right)} ,}
\end{equation}
and
\begin{equation}\label{Eq20}
{\textstyle
{}_a^{\rm{R}}{\mathscr I}_t^{-\alpha} \left\{ {f\left( t \right)g\left( t \right)} \right\}\\
= \sum\nolimits_{j = 0}^{ + \infty } {\big( {\begin{smallmatrix}
\alpha \\
j
\end{smallmatrix}} \big){g^{\left( j \right)}}\left( t \right){}_a^{\rm{R}}{\mathscr I}_t^{-\alpha  + j}f\left( t \right)} ,}
\end{equation}
respectively.
\end{remark}

Notably, many published results adopted the Leibniz rule (\ref{Eq12}) for the Caputo definition (see (2) in \cite{Aguila:2014CNSNS}, (3) in \cite{Aguila:2015CNSNS}, (7) in \cite{Hua:2018NEU}), that is to say,
\begin{equation}\label{Eq21}
{\textstyle
{}_a^{\rm{C}}{\mathscr D}_t^\alpha \left\{ {f\left( t \right)g\left( t \right)} \right\} = \sum\nolimits_{j = 0}^{ + \infty } {\big( {\begin{smallmatrix}
\alpha \\
j
\end{smallmatrix}} \big){f^{\left( j \right)}}\left( t \right){}_a^{\rm{C}}{\mathscr D}_t^{\alpha  - j}g\left( t \right)},}
\end{equation}
while almost no scholars explored whether it is true or false.

In what follows, we will analytically prove that the Leibniz rule (\ref{Eq21}) is not established unless the Caputo derivative of $g(t)$ is equal to the counterpart of Riemann--Liouville case.

\begin{proposition}\label{Proposition 1}
Leibniz rule of Caputo derivative (\ref{Eq21}) does not hold.
\end{proposition}
\begin{proof}
To illustrate the violation of (\ref{Eq21}), two aspects will be deployed. Likewise, if we put $\alpha$ in the interval $(n-1,n)$ with positive integer $n$, then a question will emerge instantly, namely, how to interpret ${}_a^{\rm{C}}{\mathscr D}_t^{\alpha  - j}g\left( t \right),~j > n$. Recalling Leibniz rule for Riemann--Liouville derivative, one directly assumes that ${}_a^{\rm{C}}{\mathscr D}_t^{\alpha  - j}g\left( t \right)\triangleq{}_a^{\rm{R}}{\mathscr I}_t^{j-\alpha}g\left( t \right)$ for $j > n$.

\begin{itemize}
  \item Conflict with the facts.
\end{itemize}

Assuming (\ref{Eq21}) holds, $g(t)=1,~t\ge a$ gives
\begin{equation}\label{Eq22}
\begin{array}{rl}
{}_a^{\rm{C}}{\mathscr D}_t^\alpha f\left( t \right) =&\hspace{-6pt} {}_a^{\rm{C}}{\mathscr D}_t^\alpha \left\{ {f\left( t \right)g\left( {t } \right)} \right\}\\
=&\hspace{-6pt}\sum\nolimits_{k = 0}^{ + \infty } {\big( {\begin{smallmatrix}
\alpha \\
k
\end{smallmatrix}} \big){f^{\left( k \right)}}\left( t \right){}_a^{\rm{C}}{\mathscr D}_t^{\alpha  - k}g\left( {t} \right)} \\
 =&\hspace{-6pt} \sum\nolimits_{k = n}^{ + \infty } {\big( {\begin{smallmatrix}
\alpha \\
k
\end{smallmatrix}} \big)\frac{{{f^{\left( k \right)}}\left( t \right)}}{{\Gamma \left( {k + 1 - \alpha } \right)}}{{\left( {t - a} \right)}^{k - \alpha }}}.
\end{array}
\end{equation}

On one hand, such a formula (\ref{Eq22}) conflicts with the validated conclusion in (\ref{Eq11}). On the other hand, revisiting the lower discontinuity of Caputo derivative to its order, it follows from (\ref{Eq22}) that
\begin{equation}\label{Eq23}
\begin{array}{rl}
\mathop {\lim }\limits_{\alpha  \to \left( {n - 1} \right) + } {}_a^{\rm{C}}{\mathscr D}_t^\alpha f\left( t \right) =&\hspace{-6pt} \sum\nolimits_{k = n}^{ + \infty } {\frac{{\Gamma \left( {n} \right)}}{{\Gamma \left( {k + 1} \right)\Gamma \left( {n-k} \right)}}\frac{{{f^{\left( k \right)}}\left( t \right)}}{{\Gamma \left( {k-n+2} \right)}}{{\left( {t - a} \right)}^{k-n+1}}} \\
 = &\hspace{-6pt}  0
\end{array}
\end{equation}
which is different from another validated conclusion
\begin{equation}\label{Eq24}
\begin{array}{rl}
\mathop {\lim }\limits_{\alpha  \to \left( {n - 1} \right) + } {}_a^{\rm{C}}{\mathscr D}_t^\alpha f\left( t \right)=&\hspace{-6pt} \sum\nolimits_{k = n}^{+\infty}  {\frac{{{f^{\left( k \right)}}\left( a \right)}}{{\Gamma \left( { k-n+2 } \right)}}} {\left( {t - a} \right)^{k-n+1 }}\\
=&\hspace{-6pt}{f^{\left( {n - 1} \right)}}\left( t \right) - {f^{\left( {n - 1} \right)}}\left( a \right)\not\equiv0.
\end{array}
\end{equation}
The incorrectness of (\ref{Eq21}) is thus concluded.

\begin{itemize}
  \item Contradict oneself.
\end{itemize}

Defining $h(t)=1$ for any $t\ge a$, one has
\begin{equation}\label{Eq25}
\hspace{-6pt}\begin{array}{l}
{}_a^{\rm{C}}{\mathscr D}_t^\alpha \left\{ {f\left( t \right)g\left( t \right)} \right\}\\
= {}_a^{\rm{C}}{\mathscr D}_t^\alpha \left\{ {\left[ {f\left( t \right)g\left( t \right)} \right]h\left( {t } \right)} \right\}\\
= \sum\nolimits_{k = 0}^{ + \infty } {\big( {\begin{smallmatrix}
\alpha \\
k
\end{smallmatrix}} \big)\frac{{{{\rm{d}}^k}}}{{{\rm{d}}{t^k}}}\left\{ {f\left( t \right)g\left( t \right)} \right\}{}_a^{\rm{C}}{\mathscr D}_t^{\alpha  - k}h\left( {t } \right)} \\
= \sum\nolimits_{k = n}^{ + \infty } {\big( {\begin{smallmatrix}
\alpha \\
k
\end{smallmatrix}} \big)\sum\nolimits_{j = 0}^k {\big( {\begin{smallmatrix}
k\\
j
\end{smallmatrix}} \big){f^{\left( j \right)}}\left( t \right){g^{\left( {k - j} \right)}}\left( t \right)} {}_a^{\rm{C}}{\mathscr D}_t^{\alpha  - k}h\left( {t } \right)} \\
\not\equiv\sum\nolimits_{k = 0}^{ + \infty } {\big( {\begin{smallmatrix}
\alpha \\
k
\end{smallmatrix}} \big)\sum\nolimits_{j = 0}^k {\big( {\begin{smallmatrix}
k\\
j
\end{smallmatrix}} \big){f^{\left( j \right)}}\left( t \right){g^{\left( {k - j} \right)}}\left( t \right)} {}_a^{\rm{C}}{\mathscr D}_t^{\alpha  - k}h\left( {t } \right)} \\
=\sum\nolimits_{j = 0}^{ + \infty } {{f^{\left( j \right)}}\left( t \right)\sum\nolimits_{k = j}^{ + \infty } {\big( {\begin{smallmatrix}
\alpha \\
k
\end{smallmatrix}} \big)\big( {\begin{smallmatrix}
k\\
j
\end{smallmatrix}} \big){g^{\left( {k - j} \right)}}\left( t \right)} {}_a^{\rm{C}}{\mathscr D}_t^{\alpha  - k}h\left( {t } \right)} \\
 =\sum\nolimits_{j = 0}^{ + \infty } {{f^{\left( j \right)}}\left( t \right)\sum\nolimits_{i = 0}^{ + \infty } {\big( {\begin{smallmatrix}
\alpha \\
{i + j}
\end{smallmatrix}} \big)\big( {\begin{smallmatrix}
{i + j}\\
j
\end{smallmatrix}} \big){g^{\left( i \right)}}\left( t \right)} {}_a^{\rm{C}}{\mathscr D}_t^{\alpha  - j - i}h\left( {t } \right)} \\
 =\sum\nolimits_{j = 0}^{ + \infty } {\left( {\begin{smallmatrix}
\alpha \\
j
\end{smallmatrix}} \right){f^{\left( j \right)}}\left( t \right)\sum\nolimits_{i = 0}^{ + \infty } {\left( {\begin{smallmatrix}
{\alpha  - j}\\
i
\end{smallmatrix}} \right){g^{\left( i \right)}}\left( t \right)} {}_a^{\rm{C}}{\mathscr D}_t^{\alpha  - j - i}h\left( {t } \right)} \\
 = \sum\nolimits_{j = 0}^{ + \infty } {\left( {\begin{smallmatrix}
\alpha \\
j
\end{smallmatrix}} \right){f^{\left( j \right)}}\left( t \right){}_a^{\rm{C}}{\mathscr D}_t^{\alpha  - j}g\left( t \right)},
\end{array}
\end{equation}
which shows the self-contradictory problem of formula (\ref{Eq21}). For $j>n$, ${}_a^{\rm{C}}{\mathscr D}_t^{\alpha  - j}$ is calculated by the way of the Riemann--Liouville integral, although there is no reason for this assumption. Apart from this queer assumption, another flaw is that the initial step $k$ equals to $n$ not $0$.
\end{proof}

Now that the well-known Leibniz rule in (\ref{Eq21}) for Caputo derivative is incorrect, it is important and interesting to derive the correct one. With the help of the correct Leibniz rule in Lemma \ref{Lemma 2} and the relation between Caputo derivative and Riemann--Liouville derivative (\ref{Eq5}), a right result can be arrived immediately.
\begin{theorem}[Leibniz rule of Caputo derivative]\label{Theorem 3}
For any constant $\alpha\in(n-1,n),~n\in\mathbb{N}_+$, if $f(t)$ can be expressed as a Taylor series and $g(t)$ is $n$-th continuously differentiable, one has
\begin{equation}\label{Eq26}
\begin{array}{l}
{}_a^{\rm{C}}{\mathscr D}_t^\alpha \left\{ {f\left( t \right)g\left( t \right)} \right\}
= \sum\nolimits_{j = 0}^{ + \infty } {\left( {\begin{smallmatrix}
\alpha \\
j
\end{smallmatrix}} \right){f^{\left( j \right)}}\left( t \right){}_a^{\rm{C}}{\mathscr D}_t^{\alpha  - j}g\left( t \right)} +R_1,
\end{array}
\end{equation}
and if $g(t)$ can be expressed as a Taylor series and $f(t)$ is $n$-th continuously differentiable, similar result follows
\begin{equation}\label{Eq27}
\begin{array}{l}
{}_a^{\rm{C}}{\mathscr D}_t^\alpha \left\{ {f\left( t \right)g\left( t \right)} \right\}
= \sum\nolimits_{j = 0}^{ + \infty } {\left( {\begin{smallmatrix}
\alpha \\
j
\end{smallmatrix}} \right){g^{\left( j \right)}}\left( t \right){}_a^{\rm{C}}{\mathscr D}_t^{\alpha  - j}f\left( t \right)} +R_2,
\end{array}
\end{equation}
where \[
\begin{array}{l}
R_1= \sum\nolimits_{k = 0}^{n - 1} {\sum\nolimits_{j = 0}^k {\left[ {\left( {\begin{smallmatrix}
\alpha \\
j
\end{smallmatrix}} \right){f^{\left( j \right)}}\left( t \right) - \left( {\begin{smallmatrix}
k\\
j
\end{smallmatrix}} \right){f^{\left( j \right)}}\left( a \right)} \right]\frac{{{g^{\left( {k - j} \right)}}\left( a \right){{\left( {t - a} \right)}^{k - \alpha }}}}{{\Gamma \left( {k + 1 - \alpha } \right)}}} },\\
R_2= \sum\nolimits_{k = 0}^{n - 1} {\sum\nolimits_{j = 0}^k {\left[ {\left( {\begin{smallmatrix}
\alpha \\
j
\end{smallmatrix}} \right){g^{\left( j \right)}}\left( t \right) - \left( {\begin{smallmatrix}
k\\
j
\end{smallmatrix}} \right){g^{\left( j \right)}}\left( a \right)} \right]\frac{{{f^{\left( {k - j} \right)}}\left( a \right){{\left( {t - a} \right)}^{k - \alpha }}}}{{\Gamma \left( {k + 1 - \alpha } \right)}}} }.
\end{array}\]

\end{theorem}
\begin{proof} Applying Lemma \ref{Lemma 2} and (\ref{Eq5}) yields
\begin{equation}\label{Eq28}
\begin{array}{l}
{}_a^{\rm C}{\mathscr D}_t^\alpha \left\{ {f\left( t \right)g\left( t \right)} \right\}\\
 = {}_a^{\rm R}{\mathscr D}_t^\alpha \left\{ {f\left( t \right)g\left( t \right)} \right\} - \sum\nolimits_{k = 0}^{n - 1} {\frac{{{{\rm d}^k}}}{{{\rm d}{t^k}}}{{\left. {\left\{ {f\left( t \right)g\left( t \right)} \right\}} \right|}_{t = a}}} \frac{{{{\left( {t - a} \right)}^{k - \alpha }}}}{{\Gamma \left( {k + 1 - \alpha } \right)}}\\
 = \sum\nolimits_{j = 0}^{ + \infty } {\left( {\begin{smallmatrix}
\alpha \\
j
\end{smallmatrix}} \right){f^{\left( j \right)}}\left( t \right){}_a^{\rm R}{\mathscr D}_t^{\alpha  - j}g\left( t \right)} \\
\hspace{10pt}- \sum\nolimits_{k = 0}^{n - 1} {\sum\nolimits_{j = 0}^k {\left( {\begin{smallmatrix}
k\\
j
\end{smallmatrix}} \right)} \frac{{{f^{\left( j \right)}}\left( a \right){g^{\left( {k - j} \right)}}\left( a \right)}}{{\Gamma \left( {k + 1 - \alpha } \right)}}} {\left( {t - a} \right)^{k - \alpha }}\\
=\sum\nolimits_{j = 0}^{ + \infty } {\left( {\begin{smallmatrix}
\alpha \\
j
\end{smallmatrix}} \right){f^{\left( j \right)}}\left( t \right){}_a^{\rm C}{\mathscr D}_t^{\alpha  - j}g\left( t \right)} \\
\hspace{10pt}+ \sum\nolimits_{j = 0}^{ + \infty} {\left( {\begin{smallmatrix}
\alpha \\
j
\end{smallmatrix}} \right){f^{\left( j \right)}}\left( t \right)\sum\nolimits_{k = 0}^{n - 1 - j} {\frac{{{g^{\left( k \right)}}\left( a \right)}}{{\Gamma \left( {k + j + 1 - \alpha } \right)}}} {{\left( {t - a} \right)}^{k + j - \alpha }}} \\
\hspace{10pt} - \sum\nolimits_{k = 0}^{n - 1} {\sum\nolimits_{j = 0}^k {\left( {\begin{smallmatrix}
k\\
j
\end{smallmatrix}} \right)} \frac{{{f^{\left( j \right)}}\left( a \right){g^{\left( {k - j} \right)}}\left( a \right)}}{{\Gamma \left( {k + 1 - \alpha } \right)}}} {\left( {t - a} \right)^{k - \alpha }}\\
 = \sum\nolimits_{j = 0}^{ + \infty } {\left( {\begin{smallmatrix}
\alpha \\
j
\end{smallmatrix}} \right){f^{\left( j \right)}}\left( t \right){}_a^{\rm C}{\mathscr D}_t^{\alpha  - j}g\left( t \right)} \\
\hspace{10pt} + \sum\nolimits_{j = 0}^{n - 1} {\left( {\begin{smallmatrix}
\alpha \\
j
\end{smallmatrix}} \right){f^{\left( j \right)}}\left( t \right)\sum\nolimits_{k = 0}^{n - 1 - j} {\frac{{{g^{\left( k \right)}}\left( a \right)}}{{\Gamma \left( {k + j + 1 - \alpha } \right)}}} {{\left( {t - a} \right)}^{k + j - \alpha }}} \\
\hspace{10pt} - \sum\nolimits_{k = 0}^{n - 1} {\sum\nolimits_{j = 0}^k {\left( {\begin{smallmatrix}
k\\
j
\end{smallmatrix}} \right)} \frac{{{f^{\left( j \right)}}\left( a \right){g^{\left( {k - j} \right)}}\left( a \right)}}{{\Gamma \left( {k + 1 - \alpha } \right)}}} {\left( {t - a} \right)^{k - \alpha }}\\
 = \sum\nolimits_{j = 0}^{ + \infty } {\left( {\begin{smallmatrix}
\alpha \\
j
\end{smallmatrix}} \right){f^{\left( j \right)}}\left( t \right){}_a^{\rm C}{\mathscr D}_t^{\alpha  - j}g\left( t \right)} \\
\hspace{10pt} + \sum\nolimits_{j = 0}^{n - 1} {\sum\nolimits_{k = j}^{n - 1} {\left( {\begin{smallmatrix}
\alpha \\
j
\end{smallmatrix}} \right)\frac{{{f^{\left( j \right)}}\left( t \right){g^{\left( {k - j} \right)}}\left( a \right)}}{{\Gamma \left( {j + 1 - \alpha } \right)}}{{\left( {t - a} \right)}^{j - \alpha }}} } \\
\hspace{10pt} - \sum\nolimits_{k = 0}^{n - 1} {\sum\nolimits_{j = 0}^k {\left( {\begin{smallmatrix}
k\\
j
\end{smallmatrix}} \right)} \frac{{{f^{\left( j \right)}}\left( a \right){g^{\left( {k - j} \right)}}\left( a \right)}}{{\Gamma \left( {k + 1 - \alpha } \right)}}} {\left( {t - a} \right)^{k - \alpha }}\\
 = \sum\nolimits_{j = 0}^{ + \infty } {\left( {\begin{smallmatrix}
\alpha \\
j
\end{smallmatrix}} \right){f^{\left( j \right)}}\left( t \right){}_a^{\rm C}{\mathscr D}_t^{\alpha  - j}g\left( t \right)} \\
\hspace{10pt} + \sum\nolimits_{k = 0}^{n - 1} {\sum\nolimits_{j = 0}^k {\left[ {\left( {\begin{smallmatrix}
\alpha \\
j
\end{smallmatrix}} \right){f^{\left( j \right)}}\left( t \right) - \left( {\begin{smallmatrix}
k\\
j
\end{smallmatrix}} \right){f^{\left( j \right)}}\left( a \right)} \right]\frac{{{g^{\left( {k - j} \right)}}\left( a \right){{\left( {t - a} \right)}^{k - \alpha }}}}{{\Gamma \left( {k + 1 - \alpha } \right)}}} } ,
\end{array}
\end{equation}
where $\sum\nolimits_{j = 0}^{ n-1 } {\sum\nolimits_{k = j}^{ n-1 } {} }=\sum\nolimits_{k = 0}^{ n-1 } {\sum\nolimits_{j = 0}^k }$ is adopted.

With the interchangeability of the product operation, it follows
\begin{equation}\label{Eq29}
{}_a^{\rm{C}}{\mathscr D}_t^\alpha \left\{ {f\left( t \right)g\left( t \right)} \right\}={}_a^{\rm{C}}{\mathscr D}_t^\alpha \left\{ {g\left( t \right)f\left( t \right)} \right\},
\end{equation}
for all possible functions $f(t)$, $g(t)$ and the order $n-1<\alpha<n\in\mathbb{N}_+$. Combining with (\ref{Eq26}) and (\ref{Eq29}), (\ref{Eq27}) can be concluded immediately. The proof is thus completed.
\end{proof}

\begin{remark}\label{Remark 2}
Notably, Theorem \ref{Theorem 3} indicates that Leibniz rule in (\ref{Eq21}) does not always hold, even when ${}_a^{\rm{C}}{\mathscr D}_t^{\alpha  - j}g\left( t \right)={}_a^{\rm{R}}{\mathscr I}_t^{j-\alpha}g\left( t \right)$ is assumed for $j > n$. Two available Leibniz rules for Caputo derivativea are developed via the series representation method and they can also be established with the help of (\ref{Eq10}) and (\ref{Eq11}). When the range of $\alpha$ is set in $(0,1)$, the proposed formula (\ref{Eq26}) degenerates into Theorem 3.17 in \cite{Diethelm:2010Book}. From the previous discussion, (\ref{Eq26}) indicates the following results.
\begin{itemize}
  \item ${}_a^{\rm{C}}{\mathscr D}_t^\alpha \left\{ {f\left( t \right)g\left( t \right)} \right\}= \sum\nolimits_{j = 0}^{ + \infty } {\left( {\begin{smallmatrix}
\alpha \\
j
\end{smallmatrix}} \right){f^{\left( j \right)}}\left( t \right){}_a^{\rm{R}}{\mathscr D}_t^{\alpha  - j}g\left( t \right)}$ holds if $f^{(j)}(a)=0$ or $g^{(j)}(a)=0$ for all $j=0,1,\cdots,n-1$.
  \item ${}_a^{\rm{C}}{\mathscr D}_t^\alpha \left\{ {f\left( t \right)g\left( t \right)} \right\}= \sum\nolimits_{j = 0}^{ + \infty } {\left( {\begin{smallmatrix}
\alpha \\
j
\end{smallmatrix}} \right){f^{\left( j \right)}}\left( t \right){}_a^{\rm{C}}{\mathscr D}_t^{\alpha  - j}g\left( t \right)}$ holds if $g^{(j)}(a)=0$ for all $j=0,1,\cdots,n-1$.
\end{itemize}

Besides, the developed results can also be extended to the $\Psi$-H fractional derivative case like \cite{Sousa:2018arXiv}. Compared with the results in Lemma 3 of \cite{Ishteva:2005MSRJ}, our results are more prompt and practical.
\end{remark}

\subsection{Laplace Transform}

With the help of the infinite series in Lemma \ref{Lemma 1}, some results on Laplace transform can be double checked.

Let us start with the elementary unit $t^\mu$, as
\begin{equation}\label{Eq30}
\begin{array}{rl}
{\mathscr L}\left\{ {{t^\mu }} \right\} = &\hspace{-6pt}\int_0^{{\rm{ + }}\infty } {{{\rm{e}}^{ - st}}{t^\mu }{\rm{d}}t} \\
 =&\hspace{-6pt} \frac{1}{{{s^{\mu  + 1}}}}\int_0^{ + \infty } {{{\rm{e}}^{ - st}}{{\left( {st} \right)}^\mu }{\rm{d}}\left( {st} \right)} \\
 =&\hspace{-6pt} \frac{1}{{{s^{\mu  + 1}}}}\int_0^{ + \infty } {{{\rm{e}}^{ - \tau }}{\tau ^\mu }{\rm{d}}\tau } \\
 =&\hspace{-6pt} \frac{{\Gamma \left( {\mu  + 1} \right)}}{{{s^{\mu  + 1}}}},
\end{array}
\end{equation}
in which the definition $\Gamma (z)  \triangleq  \int_0^{{\rm{ + }}\infty } {{x^{z - 1}}{{\rm{e}}^{ - x}}\,{\rm{d}}x} $ is adopted. Actually, the corresponding result for $\mu\in\mathbb{N}$ can be obtained via the formula of integration by parts. The range of $\mu$ should be limited as $\mu>-1$ since the integral definition of Gamma function is only suitable for positive argument. From the integrability of $t^\mu$, one has
\begin{equation}\label{Eq31}
{\textstyle
\begin{array}{rl}
{}_0^{\rm{R}}{\mathscr I}_t^\alpha {t^\mu } =&\hspace{-6pt} \frac{1}{{\Gamma \left( \alpha  \right)}}\int_0^t {{{\left( {t - \tau } \right)}^{\alpha  - 1}}{\tau ^\alpha }{\rm{d}}\tau } \\
 =&\hspace{-6pt} \frac{1}{{\Gamma \left( \alpha  \right)}}\int_0^1 {{t^{\alpha  - 1}}{{\left( {1 - u} \right)}^{\alpha  - 1}}{t^\mu }{u^\mu }t{\rm{d}}u} \\
 =&\hspace{-6pt} \frac{{{t^{\alpha  + \mu }}}}{{\Gamma \left( \alpha  \right)}}\int_0^1 {{u^\mu }{{\left( {1 - u} \right)}^{\alpha  - 1}}{\rm{d}}u} \\
 =&\hspace{-6pt} \frac{{{t^{\alpha  + \mu }}}}{{\Gamma \left( \alpha  \right)}}{\mathcal B}\left( {\mu  + 1,\alpha } \right)\\
 =&\hspace{-6pt} \frac{{\Gamma \left( {\mu  + 1} \right)}}{{\Gamma \left( {\alpha  + \mu  + 1} \right)}}{t^{\alpha  + \mu }},
\end{array}}
\end{equation}
where the integral formula of Beta function holds only for positive arguments, i.e., $\mu+1>0$ and $\alpha>0$.

Without loss of generality, let us suppose $\mu\in(-2,-1)$, then $-1<\mu+1<0$. From the differential property of Laplace transform, it follows
\begin{equation}\label{Eq32}
{\textstyle
\begin{array}{rl}
{\mathscr L}\left\{ {{t^\mu }} \right\} =&\hspace{-6pt} {\mathscr L}\big\{ {\frac{{\rm{d}}}{{{\rm{d}}t}}\frac{{{t^{\mu  + 1}}}}{{\mu  + 1}}} \big\}\\
 =&\hspace{-6pt} s{\mathscr L}\big\{ {\frac{{{t^{\mu  + 1}}}}{{\mu  + 1}}} \big\} - { {\frac{{{t^{\mu  + 1}}}}{{\mu  + 1}}} \big|_{t = 0}}\\
 =&\hspace{-6pt} \frac{{\Gamma \left( {\mu  + 1} \right)}}{{{s^{\mu  + 1}}}} - \infty,
\end{array}}
\end{equation}
which indicates that the Laplace transform of $t^\mu$ with $\mu<-1$ does not exist in the classical sense. This result might bring the singularity of ${\mathscr L}\left\{ {{}_0^{\rm R}{\mathscr D}_t^\alpha f\left( t \right)} \right\}$ for an analytical function $f(\cdot)$ with $\alpha>1$.

If one defines $F\left( s \right) \triangleq {\mathscr L}\left\{ {f\left( t \right)} \right\}$, where $f(t)$ can be expressed as a Taylor series, according to the definition, one has
\begin{equation}\label{Eq33}
\begin{array}{rl}
F\left( s \right) =&\hspace{-6pt} \int_0^{{\rm{ + }}\infty } {{{\rm{e}}^{ - st}}f\left( t \right){\rm{d}}t} \\
 =&\hspace{-6pt} \int_0^{{\rm{ + }}\infty } {{{\rm{e}}^{ - st}}\sum\nolimits_{k = 0}^{ + \infty } {\frac{{{f^{\left( k \right)}}\left( 0 \right)}}{{k!}}{t^k}} {\rm{d}}t} \\
 =&\hspace{-6pt} \sum\nolimits_{k = 0}^{ + \infty } {\frac{{{f^{\left( k \right)}}\left( 0 \right)}}{{k!}}\int_0^{{\rm{ + }}\infty } {{{\rm{e}}^{ - st}}{t^k}{\rm{d}}t} } \\
 =&\hspace{-6pt} \sum\nolimits_{k = 0}^{ + \infty } {\frac{{{f^{\left( k \right)}}\left( 0 \right)}}{{{s^{k + 1}}}}},
\end{array}
\end{equation}
where $f(t)$ can be expressed as a Taylor series at $0$.

From these discussions, the following results can be obtained
\begin{equation}\label{Eq34}
\begin{array}{rl}
{\mathscr L}\left\{ {{}_0^{\rm R}{\mathscr I}_t^\alpha f\left( t \right)} \right\} =&\hspace{-6pt} \int_0^{{\rm{ + }}\infty } {{{\rm{e}}^{ - st}}\sum\nolimits_{k = 0}^{ + \infty } {\frac{{{f^{\left( k \right)}}\left( 0 \right)}}{{\Gamma \left( {k + 1 + \alpha } \right)}}} {t^{k + \alpha }}{\rm{d}}t} \\
 =&\hspace{-6pt} \sum\nolimits_{k = 0}^{ + \infty } {\frac{{{f^{\left( k \right)}}\left( 0 \right)}}{{\Gamma \left( {k + 1 + \alpha } \right)}}\int_0^{+\infty } {{{\rm{e}}^{ - st}}{t^{k + \alpha }}{\rm{d}}t} } \\
 =&\hspace{-6pt} \sum\nolimits_{k = 0}^{ + \infty } {\frac{{{f^{\left( k \right)}}\left( 0 \right)}}{{{s^{k + 1+ \alpha}}}}} \\
 =&\hspace{-6pt} {s^{ - \alpha }}F\left( s \right),
\end{array}
\end{equation}
for $\alpha>0$ and
\begin{equation}\label{Eq35}
\begin{array}{rl}
{\mathscr L}\left\{ {{}_0^{\rm{C}}{\mathscr D}_t^\alpha f\left( t \right)} \right\} =&\hspace{-6pt} \int_0^{{\rm{ + }}\infty } {{{\rm{e}}^{ - st}}\sum\nolimits_{k = n}^{ + \infty } {\frac{{{f^{\left( k \right)}}\left( 0 \right)}}{{\Gamma \left( {k + 1 - \alpha } \right)}}} {t^{k - \alpha }}{\rm{d}}t} \\
 =&\hspace{-6pt}{\sum\nolimits_{k = n}^{ + \infty } {\frac{{{f^{\left( k \right)}}\left( 0 \right)}}{{{s^{k + 1 - \alpha }}}}} }\\
 =&\hspace{-6pt}{s^\alpha }\big[ {F\left( s \right) - \sum\nolimits_{k = 0}^{n - 1} {\frac{{{f^{\left( k \right)}}\left( 0 \right)}}{{{s^{k + 1 - \alpha }}}}} } \big] \\
 =&\hspace{-6pt} {s^\alpha }F\left( s \right) - \sum\nolimits_{k = 0}^{n - 1} {{s^{\alpha  - k - 1}}{f^{\left( k \right)}}\left( 0 \right)},
\end{array}
\end{equation}
for $n-1<\alpha<n\in\mathbb{N}_+$. Likewise, ${\mathscr L}\left\{ {\frac{{{{\rm{d}}^n}}}{{{\rm{d}}{t^n}}} f\left( t \right)} \right\}$ can also be derived in a similar way. Actually, when the infinitely differentiable condition of $f(t)$ at $t=0$ is removed, (\ref{Eq34}) and (\ref{Eq35}) still hold (see (2.242) and (2.253) of \cite{Podlubny:1999Book}).

Consider the differential property in frequency, one has
\begin{equation}\label{Eq36}
{\textstyle \begin{array}{rl}
\frac{{{{\rm{d}}^m}}}{{{\rm{d}}{s^m}}}F\left( s \right) =&\hspace{-6pt} \frac{{{{\rm{d}}^m}}}{{{\rm{d}}{s^m}}}\sum\nolimits_{k = 0}^{ + \infty } {\frac{{{f^{\left( k \right)}}\left( 0 \right)}}{{{s^{k + 1}}}}} \\
 =&\hspace{-6pt} \sum\nolimits_{k = 0}^{ + \infty } {\frac{{{f^{\left( k \right)}}\left( 0 \right)\Gamma \left( { - k} \right)}}{{{s^{k + m + 1}}\Gamma \left( { - k - m} \right)}}} \\
 =&\hspace{-6pt} \sum\nolimits_{k = 0}^{ + \infty } {{{\left( { - 1} \right)}^m}\frac{{{f^{\left( k \right)}}\left( 0 \right)\Gamma \left( {k + m + 1} \right)}}{{{s^{k + m + 1}}\Gamma \left( {k + 1} \right)}}},
\end{array}}
\end{equation}
for any $m\in\mathbb{N}$. From the representation of Riemann--Liouville integral in Lemma \ref{Lemma 1}, it follows
\begin{equation}\label{Eq37}
{\textstyle \begin{array}{rl}
{}_0^{\rm{R}}{\mathscr I}_t^\alpha \left\{f\left( t \right) {{{\left( { - t} \right)}^m}} \right\} =&\hspace{-6pt} {\left( { - 1} \right)^m}_0^{\rm{R}}{\mathscr I}_t^\alpha \big\{ {\sum\nolimits_{k = 0}^{ + \infty } {\frac{{{f^{\left( k \right)}}\left( 0 \right)}}{{k!}}} {t^{k + m}}} \big\}\\
 =&\hspace{-6pt} {\left( { - 1} \right)^m}\sum\nolimits_{k = 0}^{ + \infty } {\frac{{{f^{\left( k \right)}}\left( 0 \right)\Gamma \left( {k + m + 1} \right)}}{{k!\Gamma \left( {k + m + \alpha  + 1} \right)}}} {t^{k + m + \alpha }}.
\end{array}}
\end{equation}

Combing the above discussions in (\ref{Eq36}) and (\ref{Eq37}), a beautiful and practical result can be reached.

\begin{corollary}\label{Corollary 2}
For any $\alpha\in\mathbb{R}_+$ and $m\in\mathbb{N}$, defining $F\left( s \right) \triangleq {\mathscr L}\left\{ {f\left( t \right)} \right\}$, one has
\begin{equation}\label{Eq38}
{\textstyle
{{\mathscr L}^{ - 1}}\left\{ {\frac{1}{{{s^\alpha }}}{F^{\left( m \right)}}\left( s \right)} \right\} = {}_0^{\rm{R}}{\mathscr I}_t^\alpha \left\{f\left( t \right) {{{\left( { - t} \right)}^m}} \right\}}.
\end{equation}
\end{corollary}

Now, let us consider the general case
\begin{equation}\label{Eq39}
\begin{array}{rl}
{\mathscr L}\left\{ {{\left(t-a\right)^\mu }} \right\} = &\hspace{-6pt}\int_0^{+\infty } {{{\rm{e}}^{ - st}}{\left(t-a\right)^\mu }{\rm{d}}t} \\
 =&\hspace{-6pt} {\rm e}^{-as}\int_{-a}^{ + \infty } {{{\rm{e}}^{ - s\xi}}{{\tau}^\xi }{\rm{d}}\xi} \\
 =&\hspace{-6pt} \frac{{\rm e}^{-as}}{{{s^{\mu  + 1}}}}\int_{-as}^{ + \infty } {{{\rm{e}}^{ - \tau }}{\tau ^\mu }{\rm{d}}\tau } \\
 =&\hspace{-6pt} \frac{{\rm e}^{-as}}{{{s^{\mu  + 1}}}}\Upsilon\left(\mu+1, -as\right),
\end{array}
\end{equation}
where $\Upsilon\left(p, q\right)\triangleq \int_{q}^{ + \infty } {\tau ^{p+1}}{{{\rm{e}}^{ -\tau }}{\rm{d}}\tau },p\in\mathbb{R}_+, q\in\mathbb{R}$ and $\Upsilon\left(p, 0\right)=\Gamma\left(p\right)$.

If $f(t)$ can be expressed as a Taylor series at $a$, then one has
\begin{equation}\label{Eq40}
{\textstyle {\mathscr L}\left\{ {f\left( t \right)} \right\} = \sum\nolimits_{k = 0}^{ + \infty } {\frac{{{f^{\left( k \right)}}\left( a \right)}}{{{s^{k + 1}}}}} \frac{{{{\rm{e}}^{ - as}}\Upsilon \left( {k + 1, - as} \right)}}{{\Gamma \left( {k + 1} \right)}},}
\end{equation}
\begin{equation}\label{Eq41}
{\textstyle {\mathscr L}\left\{ {{}_a^{\rm{R}}{\mathscr I}_t^\alpha f\left( t \right)} \right\} = \sum\nolimits_{k = 0}^{ + \infty } {\frac{{{f^{\left( k \right)}}\left( a \right)}}{{{s^{k + 1 - \alpha }}}}} \frac{{{{\rm{e}}^{ - as}}\Upsilon \left( {k + \alpha  + 1, - as} \right)}}{{\Gamma \left( {k + \alpha  + 1} \right)}},}
\end{equation}
\begin{equation}\label{Eq42}
{\textstyle {\mathscr L}\left\{ {{}_a^{\rm{C}}{\mathscr D}_t^\alpha f\left( t \right)} \right\} = \sum\nolimits_{k = n}^{ + \infty } {\frac{{{f^{\left( k \right)}}\left( a \right)}}{{{s^{k + 1 - \alpha }}}}} \frac{{{{\rm{e}}^{ - as}}\Upsilon \left( {k - \alpha  + 1, - as} \right)}}{{\Gamma \left( {k - \alpha  + 1} \right)}},}
\end{equation}
where $n-1<\alpha<n\in\mathbb{N}_+$. Note that in (\ref{Eq39})-(\ref{Eq42}), though the function $\Upsilon\left(p, q\right)$ is introduced to denote the corresponding results in a closed form, they are not as concise as the zero initial instant case. To solve this problem, the generalized Laplace transform can be introduced as
\begin{equation}\label{Eq43}
{\textstyle {{\mathscr L}_a}\left\{ {f\left( t \right)} \right\} \triangleq \int_a^{ + \infty } {{{\rm{e}}^{ - s\left( {t - a} \right)}}f\left( t \right){\rm{d}}t}.}
\end{equation}
With the help of (\ref{Eq43}), one has
\begin{equation}\label{Eq44}
{\textstyle {\mathscr L}_a\left\{ {{\left(t-a\right)^\mu }} \right\} = \frac{{\Gamma \left( {\mu  + 1} \right)}}{{{s^{\mu  + 1}}}},}
\end{equation}
\begin{equation}\label{Eq45}
{\textstyle F\left( s \right)={\mathscr L}_a\left\{ {f\left( t \right)} \right\} = \sum\nolimits_{k = 0}^{ + \infty } {\frac{{{f^{\left( k \right)}}\left( a \right)}}{{{s^{k + 1}}}}},}
\end{equation}
\begin{equation}\label{Eq46}
{\textstyle {\mathscr L}_a\left\{ {{}_a^{\rm{R}}{\mathscr I}_t^\alpha f\left( t \right)} \right\} = {s^{ - \alpha }}F\left( s \right),}
\end{equation}
\begin{equation}\label{Eq47}
{\textstyle {\mathscr L}_a\left\{ {{}_a^{\rm{C}}{\mathscr D}_t^\alpha f\left( t \right)} \right\} = {s^\alpha }F\left( s \right) - \sum\nolimits_{k = 0}^{n - 1} {{s^{\alpha  - k - 1}}{f^{\left( k \right)}}\left( a \right)}.}
\end{equation}
Without loss of generality, the initial instant $a$ is also assumed to be $0$ hereafter.

When $0<\alpha<1$, $k-\alpha>-1$ hold for all $k=0,1,\cdots$. Then one has
\begin{equation}\label{Eq48}
\begin{array}{rl}
{\mathscr L}\left\{ {{}_0^{\rm R}{\mathscr D}_t^\alpha f\left( t \right)} \right\} =& \hspace{-6pt} \int_0^{{\rm{ + }}\infty } {{{\rm{e}}^{ - st}}\sum\nolimits_{k = 0}^{ + \infty } {\frac{{{f^{\left( k \right)}}\left( 0 \right)}}{{\Gamma \left( {k + 1 - \alpha } \right)}}} {t^{k - \alpha }}{\rm{d}}t} \\
 =&\hspace{-6pt} \sum\nolimits_{k = 0}^{ + \infty } {\frac{{{f^{\left( k \right)}}\left( 0 \right)}}{{{s^{k + 1- \alpha}}}}} \\
 =&\hspace{-6pt} {s^{\alpha }}F\left( s \right),
\end{array}
\end{equation}
which matches the result in (2.256) of \cite{Podlubny:1999Book}.

Supposing $1<\alpha<2,~f(0)\neq0$ and applying (\ref{Eq32}), yields
\begin{equation}\label{Eq49}
\begin{array}{rl}
{\mathscr L}\left\{ {{}_0^{\rm R}{\mathscr D}_t^\alpha f\left( t \right)} \right\} =&\hspace{-6pt} \int_0^{{\rm{ + }}\infty } {{{\rm{e}}^{ - st}}\sum\nolimits_{k = 0}^{ + \infty } {\frac{{{f^{\left( k \right)}}\left( 0 \right)}}{{\Gamma \left( {k + 1 - \alpha } \right)}}} {t^{k - \alpha }}{\rm{d}}t} \\
 =&\hspace{-6pt} \sum\nolimits_{k = 0}^{ + \infty } {\frac{{{f^{\left( k \right)}}\left( 0 \right)}}{{{s^{k + 1- \alpha}}}}}-\infty \\
 =&\hspace{-6pt} {s^{\alpha }}F\left( s \right)-\infty,
\end{array}
\end{equation}
which means that in this case the Laplace transform of Riemann--Liouville derivative is singular. Actually, the similar results can be obtained for any $n-1<\alpha<n$, $n\in\mathbb{N}_+$ and $n>1$.  In other words, the $\alpha$-th Riemann--Liouville derivative of this function $f(t)$ does not exist in the classical sense when not all of $f^{(k)}(0)$ are equal to zero for $k=0,1,\cdots,n-1$.

Recall the famous Laplace transform of Riemann--Liouville derivative
\begin{equation}\label{Eq50}
{\textstyle
{\mathscr L}\hspace{-2pt}\left\{ {{}_0^{\rm{R}}{\mathscr D}_t^\alpha f\left( t \right)} \right\}\hspace{-1pt} =\hspace{-1pt} {s^\alpha }F\left( s \right)\hspace{-1pt} - \hspace{-3pt}\sum\nolimits_{k = 0}^{n - 1}\hspace{-1pt} {{s^k}{{ {{}_0^{\rm{R}}{\mathscr D}_t^{\alpha  - k - 1}f\left( t \right)} \big|}_{t = 0}}},}
\end{equation}
in which ${{ {{}_0^{\rm{R}}{\mathscr D}_t^{\alpha  - k - 1}f\left( t \right)} \big|}_{t = 0}}$ are the needed initial conditions, including one integral initial value and $n-1$ derivative initial values \cite{Oldham:1974Book}. Note that (\ref{Eq48}) and (\ref{Eq49}) conflict with (\ref{Eq50}). Next, we will find out how the differences come into being.

Looking at equations (\ref{Eq8}) and (\ref{Eq10}), one singular property can be found on the unit $t^{-\alpha}$ when $\alpha>0$. It cannot be expressed by the bases $1,~t,~t^2,\cdots$ with finite weighting, since $\mathop {\lim }\limits_{t \to 0 + } {t^{ - \alpha }} =  + \infty $ and $\mathop {\lim }\limits_{t \to 0 + } {t^k} = 0$, $k \in \mathbb{N}$, which just shows the singularity of Riemann--Liouville derivative. If we consider the behaviour of Riemann--Liouville fractional calculus near the lower terminal, i.e., $t\to 0+$, the description
\begin{equation}\label{Eq51}
{\textstyle\begin{array}{l}
\mathop {\lim }\limits_{t \to 0 + } {}_0^{\rm{R}}{\mathscr D}_t^\alpha f\left( t \right)
= \left\{ \begin{array}{ll}
0&,~\alpha  < 0~{\rm{and}}~f\left( 0 \right)\neq\infty,\\
f\left( 0 \right)&,~\alpha  = 0,\\
\pm\infty&,~\alpha  > 0,~\alpha  \notin \mathbb{N}~{\rm{and}}~f\left( 0 \right)\neq0.
\end{array}\right.
\end{array}}
\end{equation}
follows. It clearly illustrates that the initial conditions of fractional derivative are singular in the framework of integer order. If the function $f(t)$ can be expressed as a Taylor series, the finite $f^{(k)}(0),k=0,1,\cdots,n-1$ are implied for $n-1<\alpha<n\in\mathbb{N}_+$. For $\alpha\in(0,1)$, it follows ${{ {{}_0^{\rm{R}}{\mathscr D}_t^{\alpha   - 1}f\left( t \right)} \big|}_{t = 0}}={{ {{}_0^{\rm{R}}{\mathscr I}_t^{1-\alpha  }f\left( t \right)} \big|}_{t = 0}}=0$ and then (\ref{Eq50}) reduces to (\ref{Eq48}). For $\alpha\in(1,2)$, the initial value ${{ {{}_0^{\rm{R}}{\mathscr D}_t^{\alpha   - 2}f\left( t \right)} \big|}_{t = 0}}=0$ and the initial value ${{ {{}_0^{\rm{R}}{\mathscr D}_t^{\alpha   - 1}f\left( t \right)} \big|}_{t = 0}}=\infty$ unless $f(0)=0$. As a result, the singular case in (\ref{Eq49}) appears.

Furthermore, representing the finite initial value $f^{(k)}(0)=b_k,~k=0,1,\cdots,n-1$, then $F(s)$ can be expressed as
\begin{equation}\label{Eq52}
{\textstyle F\left( s \right) = \sum\nolimits_{k = 0}^{n - 1} {\frac{{{b_k}}}{{{s^{k + 1}}}}}  + O\left( {\frac{1}{{{s^{n+1}}}}} \right).}
\end{equation}

By using the result in (\ref{Eq51}), nonzero $b_k$ lead to the singular results
\begin{equation}\label{Eq53}
{\textstyle \mathop {\lim }\limits_{t \to 0} {}_0^{\rm{R}}{\mathscr D}_t^{\alpha  - n}f\left( t \right)  = 0,}
\end{equation}
\begin{equation}\label{Eq54}
{\textstyle \mathop {\lim }\limits_{t \to 0} {}_0^{\rm{R}}{\mathscr D}_t^{\alpha  - 1 - k}f\left( t \right) = \infty ,~k = 0,1, \cdots ,n - 2.}
\end{equation}

\begin{corollary}\label{Corollary 3}
If the function $f(t)$ is $n$-th continuously differentiable, then the following two conditions are equivalent.
\begin{equation}\label{Eq55}
{\textstyle {}_0^{\rm{R}}{\mathscr D}_t^{\alpha  - 1-k}f\left( t \right) = 0,~k = 0,1, \cdots ,n - 1,}
\end{equation}
\begin{equation}\label{Eq56}
{\textstyle f^{(k)}\left( t \right) = 0 ,~k = 0,1, \cdots ,n - 1,}
\end{equation}
where $n-1<\alpha<n\in\mathbb{N}_+$.
\end{corollary}

For all zero $b_k$, $\mathop {\lim }\limits_{t \to 0} {}_0^{\rm{R}}{\mathscr D}_t^{\alpha  - 1 - k}f\left( t \right), k = 0,1, \cdots ,n - 1 $ or even $\mathop {\lim }\limits_{t \to 0} {}_0^{\rm{R}}{\mathscr D}_t^{\alpha }f\left( t \right)$ equal zero too. In this case, ${}_a^{\rm{R}}{\mathscr D}_t^\alpha f\left( t \right) = {}_a^{\rm{C}}{\mathscr D}_t^\alpha f\left( t \right)$ and their Laplace transform are also equal.

\begin{remark}\label{Remark 3}
Before closing the main results, the main contribution of this work are summarized as follows.
\begin{enumerate}[i)]
  \item The Leibniz rules of Riemann--Liouville fractional integral and derivative for $f(t)t^m$ are developed.
  \item It is proved that the Leibniz rule of Riemann--Liouville derivative does not satisfy Caputo derivative.
  \item The Leibniz rule of Caputo derivative is established with the help of the series representation of Caputo derivative.
  \item The simplified scenario and its corresponding conditions are provided for the Leibniz rule of Caputo derivative.
  \item By using the series representation, the Laplace transforms are derived for ${{}_0^{\rm R}{\mathscr I}_t^\alpha f\left( t \right)}$ and ${{}_0^{\rm C}{\mathscr D}_t^\alpha f\left( t \right)}$.
  \item Assuming $f(\cdot)$ is an analytic function, the relationship between ${\frac{1}{{{s^\alpha }}}{F^{\left( m \right)}}\left( s \right)}$ and ${}_0^{\rm{R}}{\mathscr I}_t^\alpha \left\{f\left( t \right) {{{\left( { - t} \right)}^m}} \right\}$ is built.
  \item In view of the nonzero initial instant, the Laplace transforms are derived for ${{}_a^{\rm R}{\mathscr I}_t^\alpha f\left( t \right)}$ and ${{}_a^{\rm C}{\mathscr D}_t^\alpha f\left( t \right)}$.
  \item With the generalized Laplace transform, ${\mathscr L}_a\left\{ {{}_a^{\rm{R}}{\mathscr I}_t^\alpha f\left( t \right)} \right\}$ and ${\mathscr L}_a\left\{ {{}_a^{\rm{C}}{\mathscr D}_t^\alpha f\left( t \right)} \right\}$ are deduced in concise form.
  \item The singularity is discussed for the Laplace transform of ${\mathscr L}_a\left\{ {{}_0^{\rm{R}}{\mathscr D}_t^\alpha f\left( t \right)} \right\}$ with an analytic function $f(\cdot)$.
  \item The equivalence of ${}_0^{\rm{R}}{\mathscr D}_t^{\alpha  - 1-k}f\left( t \right) = 0$ and $f^{(k)}\left( t \right) = 0 ,~k = 0,1, \cdots ,n - 1$ is discussed for $n$-th continuously differentiable $f(\cdot)$.
\end{enumerate}
\end{remark}

\section{Simulation Study}\label{Section 4}
\setcounter{section}{4} \setcounter{equation}{0}
In this section, four numerical examples are provided to illustrate the validity  and the usefulness of the proposed results. Additionally, ${}_a^{\rm{C}}{\mathscr D}_t^{-\beta}f\left( t \right)$ with $\beta>0$ is calculated by ${}_a^{\rm{R}}{\mathscr I}_t^{\beta}f\left( t \right)$ instead.

\begin{example}\label{Example1}
Provided that $0 < \alpha  < 1$, $f\left( t \right) = t - a$ and $g\left( t \right) = 1$, it follows
\begin{equation}\label{Eq57}
{\textstyle {}_a^{\rm{C}}{\mathscr D}_t^\alpha \left( {t - a} \right) = \frac{{{{\left( {t - a} \right)}^{1 - \alpha }}}}{{\Gamma \left( {2 - \alpha } \right)}}.}
\end{equation}

The regular Leibniz rule (\ref{Eq21}) corresponds
\begin{equation}\label{Eq58}
{\textstyle\begin{array}{l}
 \sum\nolimits_{k = 0}^{ + \infty } {\left( {\begin{smallmatrix}
\alpha \\
k
\end{smallmatrix}} \right){f^{\left( k \right)}}\left( t \right){}_a^{\rm{C}}{\mathscr D}_t^{\alpha  - k}g\left( t \right)}\\
= \left( {\begin{smallmatrix}
\alpha \\
0
\end{smallmatrix}} \right)\left( {t - a} \right){}_a^{\rm{C}}{\mathscr D}_t^\alpha g\left( t \right) + \left( {\begin{smallmatrix}
\alpha \\
1
\end{smallmatrix}} \right){}^{\rm R}_a{\mathscr I}_t^{1 - \alpha }g\left( t \right)\\
 =\alpha \frac{{{{\left( {t - a} \right)}^{1 - \alpha }}}}{{\Gamma \left( {2 - \alpha } \right)}},
\end{array}}
\end{equation}
which differs from (\ref{Eq57}). This proves the incorrect use of Leibniz rule (\ref{Eq21}) for Caputo derivative.

From the proposed compensation formula in Theorem \ref{Theorem 3}, one has
\begin{equation}\label{Eq59}
\begin{array}{l}
{}_a^{\rm{C}}{\mathscr D}_t^\alpha \left\{ {f\left( t \right)g\left( t \right)} \right\} \\
 = \sum\nolimits_{k = 0}^{ + \infty } {\left( {\begin{smallmatrix}
\alpha \\
k
\end{smallmatrix}} \right){f^{\left( k \right)}}\left( t \right){}_a^{\rm{C}}{\mathscr D}_t^{\alpha  - k}g\left( t \right)} + \left[ {f\left( t \right) - f\left( a \right)} \right]\frac{{g\left( a \right){{\left( {t - a} \right)}^{ - \alpha }}}}{{\Gamma \left( {1 - \alpha } \right)}}\\
= \alpha \frac{{{{\left( {t - a} \right)}^{1 - \alpha }}}}{{\Gamma \left( {2 - \alpha } \right)}} + \frac{{{{\left( {t - a} \right)}^{1 - \alpha }}}}{{\Gamma \left( {2 - \alpha } \right)}}\left( {1 - \alpha } \right)\\
 = {}_a^{\rm{C}}{\mathscr D}_t^\alpha \left( {t - a} \right).
\end{array}
\end{equation}
\end{example}
\begin{example}\label{Example2}
In a similar way, setting $0 < \alpha  < 1$, $f\left( t \right) = g\left( t \right) = t$, it leads to
\begin{equation}\label{Eq60}
{\textstyle \begin{array}{rl}
{}_a^{\rm{C}}{\mathscr D}_t^\alpha {t^2} =&\hspace{-6pt} {}_a^{\rm{C}}{\mathscr D}_t^\alpha \left[ {{{\left( {t - a} \right)}^2} + 2a\left( {t - a} \right) + {a^2}} \right]\vspace{3pt}\\
 =&\hspace{-6pt} \frac{{2{{\left( {t - a} \right)}^{2 - \alpha }}}}{{\Gamma \left( {3 - \alpha } \right)}} + 2a\frac{{{{\left( {t - a} \right)}^{1 - \alpha }}}}{{\Gamma \left( {2 - \alpha } \right)}}\\
 =&\hspace{-6pt} \frac{{2{{\left( {t - a} \right)}^{1 - \alpha }}}}{{\Gamma \left( {3 - \alpha } \right)}}\left( {t + a - a\alpha } \right).
\end{array}}
\end{equation}

By using (\ref{Eq21}), the result
\begin{equation}\label{Eq61}
\begin{array}{l}
\sum\nolimits_{k = 0}^{ + \infty } {\left( {\begin{smallmatrix}
\alpha \\
k
\end{smallmatrix}} \right){f^{\left( k \right)}}\left( t \right){}_a^{\rm{C}}{\mathscr D}_t^{\alpha  - k}g\left( t \right)}\\
 = \left( {\begin{smallmatrix}
\alpha \\
0
\end{smallmatrix}} \right)t\frac{{{{\left( {t - a} \right)}^{1 - \alpha }}}}{{\Gamma \left( {2 - \alpha } \right)}} + \left( {\begin{smallmatrix}
\alpha \\
1
\end{smallmatrix}} \right)\big[ {\frac{{{{\left( {t - a} \right)}^{2 - \alpha }}}}{{\Gamma \left( {3 - \alpha } \right)}} + \frac{{a{{\left( {t - a} \right)}^{1 - \alpha }}}}{{\Gamma \left( {2 - \alpha } \right)}}} \big]\\
 = \frac{{{{\left( {t - a} \right)}^{1 - \alpha }}}}{{\Gamma \left( {3 - \alpha } \right)}}\left( {2t + a\alpha - a\alpha^2 } \right)
\end{array}
\end{equation}
emerges, which is different from (\ref{Eq60}). After adding the compensation term in Theorem \ref{Theorem 3}, the desired result can be obtained
\begin{equation}\label{Eq62}
\begin{array}{l}
{}_a^{\rm{C}}{\mathscr D}_t^\alpha \left\{ {f\left( t \right)g\left( t \right)} \right\} \\
= \sum\nolimits_{k = 0}^{ + \infty } {\left( {\begin{smallmatrix}
\alpha \\
k
\end{smallmatrix}} \right){f^{\left( k \right)}}\left( t \right){}_a^{\rm{C}}{\mathscr D}_t^{\alpha  - k}g\left( t \right)} + \left[ {f\left( t \right) - f\left( a \right)} \right]\frac{{g\left( a \right){{\left( {t - a} \right)}^{ - \alpha }}}}{{\Gamma \left( {1 - \alpha } \right)}}\\
 = \sum\nolimits_{k = 0}^{ + \infty } {\left( {\begin{smallmatrix}
\alpha \\
k
\end{smallmatrix}} \right){f^{\left( k \right)}}\left( t \right){}_a^{\rm{C}}{\mathscr D}_t^{\alpha  - k}g\left( t \right)} + \frac{{{{\left( {t - a} \right)}^{1 - \alpha }}}}{{\Gamma \left( {3 - \alpha } \right)}}\left( {2a - 3a\alpha  + a{\alpha ^2}} \right)\\
 = {}_a^{\rm{C}}{\mathscr D}_t^\alpha {t^2}.
\end{array}
\end{equation}
\end{example}

\begin{example}
Defining $f(t)=t^p+k$ with $\alpha>0$, $p-\alpha>-1$, its Riemann--Liouville derivative follows
\begin{equation}\label{Eq63}
{\textstyle {{}_0^{\rm{R}}{\mathscr D}_t^\alpha f\left( t \right)} =\frac{\Gamma(p+1)}{\Gamma(p-\alpha+1)}t^{p-\alpha}+k\frac{1}{\Gamma(1-\alpha)}t^{-\alpha}.}
\end{equation}

By applying the formula (\ref{Eq31}), one has ${\mathscr L}\left\{  f\left( t \right) \right\} =\frac{{\Gamma \left( {p + 1} \right)}}{{{s^{p  + 1}}}} + \frac{k}{{{s}}}$. Only when $0<\alpha<1$, one has
\begin{equation}\label{Eq64}
{\textstyle {\mathscr L}\left\{ {{}_0^{\rm{R}}{\mathscr D}_t^\alpha f\left( t \right)} \right\} =\frac{{\Gamma \left( {p + 1} \right)}}{{{s^{p - \alpha  + 1}}}} + \frac{k}{{{s^{ - \alpha  + 1}}}},}
\end{equation}
and the following relation holds
\begin{equation}\label{Eq65}
{\textstyle {\mathscr L}\left\{ {{}_0^{\rm{R}}{\mathscr D}_t^\alpha f\left( t \right)} \right\}=s^\alpha{\mathscr L}\left\{  f\left( t \right) \right\},}
\end{equation}
where no initial condition exists.
\end{example}

\section{Conclusions}\label{Section 5}
In this paper, two essential applications of fractional calculus have been studied by introducing the series representation. Specifically, it has proven that the commonly used Leibniz rule is not applicable to Caputo definition and the well-known Laplace transform of Riemann--Liouville is conditional. These concluded points confirm the fact that many recently published results suffer from an incorrect use of such questionable results. More importantly, the correct form of Leibniz rule and the exact condition of Laplace transform are subsequently deduced. Finally, three examples are presented to illustrate the applicability and efficiency of the presented tools. It is believed that the proposed methods indeed open a new way to analyze and solve the related problems.


\section*{Disclosure Statement}
No potential conflict of interest was reported by the authors.

\section*{Acknowledgments}
The authors would like to thank the Associate Editor and the anonymous reviewers for their keen and insightful comments which greatly improved the contents and presentation. The work described in this paper was supported by the National Natural Science Foundation of China (61601431, 61573332, 61973291), the Anhui Provincial Natural Science Foundation (1708085QF141) and the fund of China Scholarship Council (201806345002).
\bibliographystyle{tfnlm}
\bibliography{database}

\end{document}